\documentclass[12pt]{amsart}

 \usepackage{amsmath}
 \usepackage{amssymb}
 \usepackage{url}
 \usepackage{mathpazo}
 \usepackage{fullpage}
 \usepackage{color}
 \usepackage[normalem]{ulem}

 \newcommand{\unlike}{{\operatorname{unl}}}
 \newcommand{\tp}{\operatorname{tp}}
\newcommand{\ft}{{\operatorname{ft}}}
\newcommand{\trdeg}{\operatorname{tr.deg}}
\newcommand{\atyp}{\operatorname{atyp}}

\newcommand{\bB}{{\mathbf{B}}}
\newcommand{\bG}{{\mathbf{G}}}
\newcommand{\bH}{{\mathbf{H}}}
 
\newcommand{\CC}{\mathbb{C}}
\newcommand{\QQ}{\mathbb{Q}}
\newcommand{\RR}{\mathbb{R}}
\newcommand{\ZZ}{\mathbb{Z}}

\newcommand{\sS}{\mathsf{S}}

\renewcommand{\ft}{{\operatorname{ft}}}

\newcommand{\cA}{{\mathcal{A}}}
\newcommand{\cH}{{\mathcal{H}}}
\newcommand{\cM}{{\mathcal{M}}}
\newcommand{\cS}{{\mathcal{S}}}

 \newtheorem{thm}{Theorem}[section]
 \newtheorem{lem}[thm]{Lemma}
 \newtheorem{prop}[thm]{Proposition}
 \newtheorem{cor}[thm]{Corollary}
 \newtheorem{conj}[thm]{Conjecture}
 \newtheorem{fact}[thm]{Fact}
 
 \theoremstyle{definition}
 \newtheorem{Def}[thm]{Definition}
 \newtheorem{conv}[thm]{Convention}
 
\theoremstyle{remark}
\newtheorem{Rk}[thm]{Remark}

\begin{document}
 \title{Effective transcendental Zilber-Pink for 
 variations of Hodge structures}
 \author{Jonathan Pila}
 \address{University of Oxford \\
 Mathematical Institute \\
 Andrew Wiles Building \\
 Radcliffe Observatory Quarter \\
 Woodstock Rd \\
 Oxford OX2 6GG \\
 UK}
 \email{pila@maths.ox.ac.uk}
 \author{Thomas Scanlon}	 
 \thanks{T.S. is patially supported
  by NSF grant DMS-1800492.  
  T.S.'s visit to Oxford was 
  made possible by a Simon's 
  Fellowship and the hospitality
  of the Mathematical Institute of 
  the University of Oxford.}
 \address{University of California, Berkeley \\
 Department of Mathematics \\
 Evans Hall \\
 Berkeley, CA 94720-3840 \\
 USA}
 \email{scanlon@math.berkeley.edu}
 \begin{abstract}
 We prove function field versions of the 
 Zilber-Pink conjectures for varieties supporting
 a variation of Hodge structures.  A form of these
 results for Shimura varieties in the context of 
 unlikely intersections is the following.
 Let $S$ be a connected pure Shimura variety
 with a fixed quasiprojective embedding.  
 We show that there is an explicitly computable
 function $B$ of two natural number arguments 
 so that for any field extension $K$ of the complex
 numbers and Hodge generic irreducible proper subvariety 
 $X \subsetneq S_K$,  the 
 set of nonconstant points in the intersection of
 $X$ with the union of all special subvarieties
 of $X$ of dimension less than the codimension of 
 $X$ in $S$ is contained in a 
 proper subvariety of $X$ 
 of 
 degree bounded by $B(\operatorname{deg}(X),\dim(X))$.
 Our techniques are differential algebraic and rely on 
 Ax-Schanuel functional transcendence 
 theorems.  We use these results to show that the 
 differential equations associated with Shimura 
 varieties give new examples of minimal, and
 sometimes, strongly minimal, types with trivial
 forking geometry but non-$\aleph_0$-categorical
 induced structure.
\end{abstract}
\maketitle

\section{Introduction}
 On general grounds, if $Y$ and $Z$ are irreducible 
subvarieties of the smooth variety $X$, then 
for each component $U$ of $Y \cap Z$, we have 
$\dim(U) \geq \dim(Y) + \dim(Z) - \dim(X)$ and 
we \emph{expect} to have actual equality.  We say 
that $U$ is an \emph{atypical} component of the 
intersection if its dimension is strictly 
greater than what is expected.   In the case that 
$\dim(Y) + \dim(Z) < \dim(X)$, then with the 
expected dimension of the intersection being negative, 
we are saying that we do not expect $U$ to exist at all!
In this case, we say that $U$ is an \emph{unlikely} component. 
Specializing to the case that $X$ is a Shimura variety and 
$Y \subseteq X$ is a subvariety of $X$, we define the 
atypical locus in $Y$, $Y_{\atyp}$, to be the union of all 
atypical components of intersections $Y \cap Z$ with 
$Z \subseteq X$ being a special subvariety~\footnote{We recall the precise definitions
of the special and weakly special varieties in Section~\ref{conventions}.} 
and we define the unlikely locus, $Y_{\unlike}$, to be 
the union of the unlikely components of such intersections. 

The  Zilber-Pink conjecture, in Pink's formulation, restricted to this case of 
pure Shimura varieties predicts that,  if $Y$ is not contained in a proper special subvariety of $X$, then
$Y_{\unlike}$ is not Zariski dense in $Y$.  
One could consider an effective strengthening of the Zilber-Pink 
conjecture.  Fix a quasi-projective embedding of $X$ so that the notion
of the degree of a subvariety becomes meaningful.   For these effective
versions of Zilber-Pink, we would ask 
that if $Y \subseteq X$ is not contained in a proper 
special subvariety of $X$, then there is a proper 
subvariety $Z \subsetneq Y$ with 
$Y_\unlike \subseteq Z$ having $\deg(Z)$  
bounded by an explicitly computable
function of the degree and the dimension of $Y$.

In this note we prove a function field version of this effective Zilber-Pink  conjecture.  
Let $K$ be a field extension of the complex numbers.  
 We shall show that for $Y \subsetneq X_K$ a proper, irreducible subvariety of $X$
defined over $K$, not contained in a proper special variety, there is some subvariety 
$Z \subsetneq Y$ which contains all of the nonconstant points in $Y_\unlike$
having $\deg(Z)$ bounded by an explicitly computable function of the 
degree and dimension of $Y$.  Indeed, our methods 
apply more generally, for instance, to variations of Hodge structures.
In the special case of powers of the modular curve the result is
as follows (see Corollary 3.8).

\begin{cor}
\label{effective-ZPY}  
Let $n$ be a positive integer and $S=Y(1)^n$.
Then there is an explicit constant $C=C(n)$ so that for any 
natural number $\ell$ and  
any irreducible subvariety $X \subseteq S_K$
with $\dim(X) + \ell < n$,
there is a proper subvariety $Y \subsetneq X$ with   
$(X(K) \smallsetminus X(\CC)) \cap \cS_S^{[\ell]} \subseteq Y$ and
$\deg(Y) \leq C \deg(X)^{\dim(S)}$. 
\end{cor}

We approach this problem by relating it to a counting problem for algebraic differential equations through
finding  non-linear algebraic differential equations satisfied by all of the points in $Y_\unlike$.  
More precisely, we deal initially with a smaller set $Y_\unlike^\ft$
of ``fully transcendental'' unlikely points which do not come from
constant-parameter points of families of weakly special subvarieties;
see 2.9, 2.13, 3.4, and 4.3.
By exploiting known functional transcendence theorems in 
the style of Ax-Schanuel and model theoretic arguments
we show that these algebraic differential equations 
cut out a differential algebraic subset which is 
not Zariski dense.  Results on effective bounds for 
the degrees of differential algebraic sets ~\cite{Bin, HrPi} provide explicit bounds on the 
degree of the Zariski closure of the set of solutions
to these differential equations.   
Our techniques apply to many other problems and 
our key technical result may be seen as a conditional 
theorem to the effect that an effective 
Zilber-Pink theorem may be deduced from a suitable 
Ax-Schanuel theorem for a given class of varieties. 

In~\cite{As19,As18}, Aslanyan proves uniform 
Zilber-Pink-type theorems for products of modular curves.
Our methods are similar in places, 
and there is some overlap in the results. 
However the motivations are quite different.
Aslanyan was motivated by extensions of the Zilber-Pink conjecture in the modular case
to include derivatives, and hence is concerned with 
varieties defined over the constants in that setting. Our objective is
to consider the Zilber-Pink conjecture itself in a function field, with varieties defined over the function field, and in general Shimura varieties (and even variations of Hodge structures), with a view getting 
effective results following~\cite{FrSc}.
In~\cite{As18},
the modular Ax-Lindemann-Weierstrass theorem with 
derivatives~\cite{Pi13} as expressed differentially 
algebraically is used to deduce uniformities from the 
compactness theorem of first-order logic.  When specialized
to this case, the main differences between the results of the 
present paper and those of~\cite{As18} are that we 
work with algebraic varieties over function fields so that 
their zero-dimensional and more generally non-strongly 
atypical unlikely intersections are controlled (while these 
are necessarily ignored in~\cite{As18}) and we 
use effective elimination theory
to give explicit upper 
bounds rather than the compactness
theorem to show the existence of 
bounds.   
Nevertheless, with Theorem~\ref{thm:constant-to-qual} we show that 
the qualitative version of our theorem  
describing unlikely intersections follows formally from 
uniform versions of the Zilber-Pink conjecture in the style
of Aslanyan's by using results of Chatzidakis, Ghioca, 
Masser and Maurin~\cite{CGMM} on unlikely intersections for 
pairs of fields.

Differential algebraic arguments of the kind we use here
were employed in~\cite{FrSc} to identify definable
sets relative to the theory of differentially closed
fields exhibiting hitherto unobserved behaviors.   
In Section~\ref{trivial-types} we extend this analysis 
to those differential equations associated to covering 
maps of Shimura varieties.  In particular, we show that for 
irreducible Shimura varieties, the associated differential
equations give new examples of strongly minimal, 
geometrically trivial, non-$\aleph_0$-categorical 
types relative to $\operatorname{DCF}_0$.   Interestingly,
by considering the differential equations associated to
families of special subvarieties, we produce examples of 
types whose Lascar and Morley ranks differ.

\section{Conventions, notation and basic definitions}
\label{conventions}

In this section we recall some basic notions and set our
notation.

We follow a notation similar to that of~\cite{BKT} to 
speak of double coset spaces, though with the 
next definition, we allow for somewhat greater 
generality.

\begin{Def}
\label{def:quotient-space}
Let $G$ be a connected real Lie group,
$M \leq G$  a compact subgroup and $\Gamma \leq G$
a discrete subgroup. Then $S_{\Gamma,G,M} := 
\Gamma \backslash G / M$ is regarded (for now) as a 
real analytic space.  We call 
$S_{\Gamma,G,M}$ a \emph{quotient space}.

Given such triples $(G,M,\Gamma)$ 
and $(G',M',\Gamma')$, a map of Lie groups 
$\varphi:G' \to G$ and an element $g \in G$ for which 
$\varphi(\Gamma') \leq \Gamma^g := g \Gamma g^{-1}$
and $\varphi(M') \leq M$, 
the function $G' \to G$ given by 
$x \mapsto g \varphi(x)$ induces a map of analytic 
spaces $S_{\Gamma',G',M'} \to S_{\Gamma,G,M}$.  
These are the morphisms between quotient spaces.      	
\end{Def}

For our purposes, we will be most concerned with
the case that $S_{\Gamma,G,M}$ has the structure of a 
complex analytic space.

\begin{Def}
\label{complex-quotient}
Let $\bG$ be a connected real algebraic group and 
$\bB \leq \bG$ an algebraic subgroup.  Set 
$G := \bG(\RR)^+$, the connected component of the
identity with respect to the Euclidean topology and 
suppose that $M := \bB(\RR) \cap G$ is compact, 
$\check{D} := \bG/\bB$ is an algebraic 
variety~\footnote{For some definability results, it is necessary to 
require $\check{D}$ to be projective.}, and that $D := G/K \subseteq \check{D}(\CC)$
is an open domain.  For $\Gamma \leq G$ discrete, the 
quotient map $q:D \to \Gamma \backslash D = \Gamma 
\backslash G / M = S_{\Gamma,G,M}$ gives 
$S_{\Gamma,G,M}$ the structure of a complex analytic 
space.  In general, we say that $S_{\Gamma,G,M}$ is a 
\emph{complex quotient space} if it arises in this manner.  Note 
that if $S_{\Gamma',G',M'}$ and $S_{\Gamma,G,M}$ are
complex quotient spaces and $\rho:S_{\Gamma',G',M'}
\to S_{\Gamma,G,M}$ is a map of quotient spaces, then 
it is complex analytic.
\end{Def}

\begin{conv}
For us, the word \emph{definable} means 
definable in the o-minimal structure  
$\mathbb{R}_{\text{an},\exp}$, 
the real field augmented by restricted 
analytic functions and the real exponential 
function. 
\end{conv}  

\begin{Def}
\label{bi-algebraic-def}
Let $S_{\Gamma,G,M}$ be a complex quotient 
space, $S$ a quasi-projective 
complex algebraic variety and 
$f:S \to S_{\Gamma,G,M}$ a map from $S$ (regarded
as a complex analytic space) to $S_{\Gamma,G,M}$. 
We say that this map is \emph{definably bi-algebraic} 
(or just \emph{bi-algebraic})
if there is a definable fundamental domain 
$F \subseteq D$ for which the fibre product
$\{ (a,b) \in F \times S(\CC) ~:~ q(a) = f(b) \}$
is definable.	
\end{Def}

\begin{Rk}
If $f:S \to S_{\Gamma,G,M}$ is definably bi-algebraic,
then by the o-minimal definable Chow theorem~\cite{PS-Crelle} 
the fibre equivalence relation on $S \times S$ given by 
$x \sim y :\Longleftrightarrow f(x) = f(y)$ is 
algebraically constructible.  Thus, at the cost 
of replacing $S$ with the constructible quotient
$S/\sim$, we may assume that $f:S \to S_{\Gamma,G,M}$ 
is an inclusion.  From now on, we tacitly make this 
assumption and regard $S$ as a (not necessarily closed) 
analytic subvariety of $S_{\Gamma,G,M}$.	
\end{Rk}

\begin{Def}
If $G = \bG(\RR)^+$ where $\bG$ is a 
semisimple $\QQ$-algebraic group,
and $\Gamma \leq G$ is an arithmetic lattice, then 
we call $S_{\Gamma,G,M}$ \emph{an arithmetic quotient}.
\end{Def}

The main theorems, 
Theorems 1.1 and 1.3   of~\cite{BKT}, express the senses in 
which arithmetic quotients are definable and period
mappings associated to variations of Hodge structures
are bi-algebraic.  We summarize these results with the 
following theorem.

\begin{fact}[\cite{BKT}]
\label{def-period}
Each arithmetic quotient $S_{\Gamma,G,M}$ is 
definable, even relative to just $\RR_{\text{alg}}$,
the ordered field of real numbers.  Relative to this 
definable structure, each morphism of 
arithmetic quotients $S_{\Gamma',G',M'} \to 
S_{\Gamma,G,M}$ is definable.    

If $S$ is an irreducible 
quasi-projective complex algebraic 
variety supporting a polarized 
variation of Hodge structure
$\mathbb{V} \to S$ of some 
fixed weight $k$, 
$\bG$ is the associated adjoint 
$\mathbb{Q}$-semisimple group of the  
generic Mumford-Tate group 
$\mathbf{M}\mathbf{T}(\mathbb{V})$, $D = G/M$ is 
its associated Mumford-Tate domain, and $\Gamma \leq G$
is the image of the monodromy representation, then
the period mapping 
$\Phi_S:S \to \Gamma \backslash D = S_{\Gamma,G,M}$
is definably bi-algebraic.	
\end{fact}

Of particular interest to us, is the case that 
$S$ is a Shimura variety and 
$\Phi_S = \operatorname{id}_S:S \to S_{\Gamma,G,M}$ 
expresses $S$ as a locally symmetric space.

The main results of~\cite{BBKT20} extend Fact~\ref{def-period} 
to show that the period map associated to a 
variation of mixed Hodge stuctures is definable.  
There is an interesting subtlety in this theorem in that
it is necessary to make a choice of the 
definable structure on the space $\Gamma \backslash D$.

\begin{Def}
A weakly special subvariety of the complex 
quotient variety $S_{\Gamma,G,M}$ is a 
an analytic subvariety of $S_{\Gamma,G,M}$
obtained as the image of a map 
$S_{\Gamma',G',M'} \to S_{\Gamma,G,M}$ of 
complex quotient varieties.	
\end{Def}

From our definition of weakly special varieties, for 
any complex quotient $S_{\Gamma,G,M}$ and point 
$x \in S_{\Gamma,G,M}$, the zero dimensional 
space $\{ x \}$ is a weakly special variety. 
More generally, for any complex quotient
$S_{\Gamma',G',M'}$, the space 
$\{ x \} \times S_{\Gamma',G',M'} 
\subseteq S_{\Gamma,G,M} \times S_{\Gamma',G',M'} 
= S_{\Gamma \times \Gamma',G \times G', M \times M'}$
is a weakly special subvariety.  
We describe this construction and isolate those 
special subvarieties which come from such horizontal 
special varieties with the next definition.

\begin{Def}
\label{def:weakly-special}
Let $T$ be a complex analytic space.  Given complex quotient spaces 
$S_{\Gamma,G,M}$, $S_{\Gamma_1,G_1,M_1}$, and 
$S_{\Gamma_2,G_2,M_2}$, a point $x \in (S_{\Gamma_1,G_1,M_1})_T$, the base change of 
$S_{\Gamma_1,G_1,M_1}$ to $T$, 
and a finite map of quotients $\rho:S_{\Gamma_1 \times \Gamma_2,G_1 \times G_2,M_1 \times M_2} \to 
S_{\Gamma,G,M}$, 
the variety $\rho(\{ x \} \times (S_{\Gamma_2,G_2,M_2})_T) 
\subseteq (S_{\Gamma,G,M})_T$ is called \emph{$T$-weakly special}.   
When $T$ is a single point with the usual reduced structure (so there has 
been no base change) and $\dim S_{\Gamma_1,G_1,M_1} \geq 1$, we 
call $\rho(\{ x \} \times S_{\Gamma_2,G_2,M_2})$ \emph{semiconstant}.
A special subvariety which is not semiconstant is called 
\emph{strongly special}.	   
\end{Def}

\begin{Rk}
In practice, $T$ will be the analytic spectrum of some 
subfield $K$ of the field $\cM$ of germs of meromorphic functions at some point on the 
complex plane. 	In this case, we will say ``$K$-weakly special'' or ``$T$-weakly special''.
\end{Rk}

\begin{Rk}
Definition~\ref{def:weakly-special} degenerates in the case that $\dim S_{\Gamma,G,M} = 0$.	
\end{Rk}

With the definition of the strongly special subvarieties in place, we
define the strongly special loci.

\begin{Def} 
Let $S_{\Gamma,G,M}$ be a complex quotient.
For each positive integer  $\ell \leq \dim S_{\Gamma,G,M}$ we 
let $\cS^{[\ell]} = \cS^{[\ell]}_{S_{\Gamma,G,M}}$ be the union 
of all strongly special subvarieties of $S_{\Gamma,G,M}$ of
dimension $\ell$.  
We set $\cS^{[\leq \ell]} := \bigcup_{i=1}^\ell  \cS^{[i]}$.

If $f:S \to S_{\Gamma,G,M}$ is definably bi-algebraic, then we define 
$\cS_S^{[\ell]} := f^{-1} \cS^{[\ell]}$ and 
$\cS_S^{[ \leq \ell ]} := f^{-1} \cS^{[\leq \ell]}$, 
each of which is a countable union of complex algebraic 
subvarieties of $S$. 
\end{Def}

\begin{Def}
Let $f:S \to S_{\Gamma,G,M}$ be bi-algebraic
and let $K$ be a $\CC$-algebra. We say that a $K$-rational
point $a \in S(K)$ is \emph{semiconstant} if there is
some semiconstant weakly special subvariety 
$Y \subseteq S_{\Gamma,G,M}$ for which $a \in f^{-1}(Y) (K)$.
Otherwise, we say that $a$ is \emph{fully transcendental}.  	
We write $S(K)^\ft$ for the set of fully transcendental points in 
$S(K)$.
\end{Def}

\begin{Rk}
If $Y$ is a weakly special subvariety of $S_{\Gamma,G,M}$, then 
$f^{-1}(Y) \subseteq S$ is an \emph{algebraic} subvariety of $S$
so that it makes sense to evaluate its set of $K$-rational points.  	
\end{Rk}
 
With these definitions in place we may state our 
functional version of the Zilber-Pink conjecture 
for unlikely intersections  in 
a qualitative form.

\begin{thm}
\label{ZP-qualitative-unlikely}
Let $\Phi_S:S \to S_{\Gamma,G,M}$ be a period mapping 
associated to a  variation of Hodge structures,
$\ell$ be a positive integer with  $\ell + \dim(S) < \dim(S_{\Gamma,G,M})$, 
$K$ be a $\CC$-algebra, and $X \subseteq S_K$ be an absolutely irreducible 
subvariety of base change of $S$ to $K$ for which $f(X)$ is not contained in 
any proper special subvariety,  then $(X(K) \smallsetminus X(\CC)) \cap \cS_S^{[\ell]}$ is not
Zariski dense in $X$. 
\end{thm}

\begin{Rk}
In~\cite{K-HL}, the corresponding Zilber-Pink conjecture (without the restriction 
to fully transcendental points) is expressed with a weaker dimension
theoretic condition. The version in Theorem~\ref{ZP-qualitative-unlikely} is an 
artifact of the statement of the existing Ax-Schanuel theorem for 
PVHS.   	
\end{Rk}

Let us recall the construction of prolongation spaces and how these
are used to describe algebraic differential equations.  
Let $(K,\partial)$ be a differential field with field of constants 
$K^\partial := \{ a \in K ~:~ \partial(a) = 0 \}$ equal to 
$\CC$.  Let $m \in \mathbb{N}$ be a natural number.  Then we have
two $K$-algebra structures on $K[\epsilon]/(\epsilon^{m+1})$, 
one coming from the usual inclusion $\iota:K \hookrightarrow  	K[\epsilon]/(\epsilon^{m+1})$
and the other coming from exponentiating the distinguished 
derivation:  $\exp(\epsilon \partial):K \to K[\epsilon]/(\epsilon^{m+1})$
given by $a \mapsto \sum_{j=0}^m \frac{1}{j!} \partial^j(a) \epsilon^j$.  
The map $\iota$ may be seen as the exponential of the trivial derivation.  

\begin{Def}
Let $(K,\partial)$ be a differential field with field of
constants $\CC$ and $m \in \mathbb{N}$ be a natural number.  
For $X$ a $K$-scheme, the $m^\text{th}$ prolongation space 
$\tau_m X$ is the $K$-scheme which represents the functor
$T \mapsto (X \otimes_{K,\exp(\epsilon \partial)} K[\epsilon]/(\epsilon^{m+1})) 
(T \otimes_{K,\iota} K[\epsilon]/(\epsilon^{m+1}))$. 	

For any differential $K$-algebra $(R,\partial_R)$, there is a
map $\nabla_m:X(R) \to \tau_m X(R)$ corresponding to the 
map of sets $X(R) \to (X \otimes_{K,\exp(\epsilon \partial)} K[\epsilon]/(\epsilon^{m+1}))
(R[\epsilon]/(\epsilon^{m+1})$ given by post-composition with $\exp(\epsilon \partial_R)$.
\end{Def}

\begin{Rk}
The construction of $\tau_m$ is 
functorial.  If $X$ is obtained from a $\mathbb{C}$-scheme by base change, then 
$\tau_m X$
is usually referred to as a \emph{jet scheme} or \emph{truncated arc space}.   We 
prefer to use the language of arc spaces and will write this as $\cA_m X$.  
We drop the subscript $m$ when it
is understood. 
\end{Rk}

\begin{Def}
If $(K,\partial)$ is a differential field, and 
$X$ is a $K$-scheme, then a differential subscheme $V$ of $X$ is 
given by a subscheme $V_m \subseteq \tau_m X$ of a prolongation space
$\tau_m X$ for some natural number $m$. If $(R,\partial_R)$ is
 a differential $K$-algebra, then $V(R) := \{ a \in X(R) ~:~ 
 \nabla_m(a) \in V_m(R) \}$.  A finite Boolean combination of
 differential subschemes of $X$ is called a \emph{differential
 constructible subset of $X$}.   A differential constructible
 $f$ function on $X$ to the $K$-scheme $Y$ is given by a 
 differential constructible subset of $X \times Y$ which 
 when evaluated on any differential $K$-algebra is the graph of a 
 function.   
\end{Def}

The equations giving the constants determine 
a particularly important class of differential
varieties.

\begin{Def}
Let $(K,\partial)$ be a differential field and 
let $X$ be a scheme over $K^\partial$, the constants
of $K$.   The constant part of $X$, 
written $X^\partial$, is the differential subscheme of
$X$ defined by $X_1^0 \subseteq \tau_1 X = \cA_1 X = TX$ where 
$X_1^0$ is the image of the zero section of $X$ inside its tangent
bundle, which in this case that $X$ is defined over the constants, 
may be identified with the first prolongation space $\tau_1 X$.  	
\end{Def}

We shall
make use of the Seidenberg embedding theorem~\cite{Se58,Se69} in the 
form that if $K \subseteq \mathcal{M}(U)$ is a differential 
subfield of the field $\mathcal{M}(U)$ of germs of meromorphic 
functions on some connected open domain $U$ containing $0$, $K$ is finitely 
generated as a differential field over $\mathbb{C}$,
and $L$ is a countably generated as 
a differential field extension of $L$, then 
$L$ may be realized as a sub differential field of
the field $\cM$ of germs at some point in $U$
of meromorphic functions on the disc over the 
natural embedding $K \hookrightarrow \mathcal{M}$.

At various point we will make use of differential algebra in the sense of
Ritt and Kolchin.  See~\cite{Ko73} for details.

\begin{Rk}
Thus our results can be equally stated in terms of a field extension 
$K$ of the complex numbers, a field  $\mathcal{M}$ of germs
of meromorphic functions as above, or of a finitely generated
field $K$ where ``constant'' means ``algebraic''.
\end{Rk}

\section{Differential equations for special 
subvarieties}
\label{DEs}

The special subvarieties of complex quotient 
spaces are themselves images of homogeneous 
spaces.   Using the notion of the generalized
Schwarzians as developed in~\cite{Sc-ade} and
then expanded in~\cite{MPT}, we may recognize
these homogeneous spaces using algebraic 
differential equations. The theorem on generalized
logarithmic derivatives of~\cite{Sc-ade} permits
us to see all of the special varieties in bi-algebraic
varieties in terms of finitely many algebraic
differential equations.

Let us recall the construction of the generalized 
Schwarzians. We are given an algebraic group $\bG$ over
$\CC$ and an action $\bG \curvearrowright X$ of $\bG$ on 
the algebraic variety $X$.  For each $m \in \mathbb{N}$, 
this action induces an action $\cA_m \bG \curvearrowright \cA_m X$ 
of the $m^\text{th}$ arc space of $\bG$ (which is itself
an algebraic group) on the $m^\text{th}$ arc space of 
$X$.   Via the section $s:\bG \to \cA_m \bG$, we obtain an action 
$\bG \curvearrowright \cA_m X$.   The quotient $\bG \backslash \cA_m X$
might not be an algbebraic variety, but it is a constructible set.   
For any differential field $(K,\partial)$ with field of constants 
$\mathbb{C}$, we may consider the $\bG(\CC)$-orbit equivalence relation 
on $X(K)$.  That is, for $a, b \in X(K)$ we have $a \sim b$ just in 
case there is some $g \in \bG(\CC)$ with $g \cdot a  = b$. 
By~\cite[Proposition 3.9]{Sc-ade}, if $m$ is large enough, 
then $\sS_{\bG,X}:X(K) \to (\bG \backslash \cA_m X)(K)$ given by sending 
$a$ to the image of $\nabla_m(a)$ in $(\bG \backslash  \cA_m X)(K)$ 
has the property that $\sS_{\bG,X}(a) = \sS_{\bG,X}(a)$ if and only 
if $a \sim b$.   We refer to $S_{\bG,X}$ as the \emph{generalized 
Schwarzian} associated to the action $\bG \curvearrowright X$ and 
to $\bG \backslash \cA_m X$ as the \emph{Schwarzian variety}.   
By making use of partial differential operators, better control on 
$m$ may be attained (see~\cite{MPT}).

Consider now a bi-algebraic variety $f:S \to S_{\Gamma,G,M}$, following the notation of 
Definitions~\ref{complex-quotient} and~\ref{bi-algebraic-def}.  By the main theorem 
of~\cite{Sc-ade} (it is stated there in the case that
$f = \operatorname{id}_S$, but the proof goes 
through whenever $f:S \to S_{\Gamma,G,M}$ 
is bi-algebraic), the ostensibly differenital 
analytically constructible function 
$\chi := \sS_{\bG,\check{D}} \circ q^{-1}$ is 
differentially constructible.  

We now use these constructions to capture the 
special varieties.

\begin{lem}
\label{H-orbit-eq}
Let $f:S \to S_{\Gamma,G,M}$ be bi-algebraic and 
$\bH \leq \bG$ be an algebraic subgroup of $\bG$.  
There is a differentially constructible set 
$\Xi_\bH \subseteq S$ defined over $\CC$ having the 
property that for any point $a \in S(\cM)$, 
we have $a \in \Xi_\bH(\cM)$ if and only if there is some 
$\widetilde{a} \in \check{D}(\cM)$, $b \in \check{D}(\CC)$,
$g \in \bG(\mathbb{C})$, and $h \in \bH(\cM)$ with 
$q(\widetilde{a}) = f(a)$ and $\widetilde{a} = h^g \cdot b = g h g^{-1} \cdot b$.
\end{lem}
\begin{proof}
Let $\widetilde{\Xi}_\bH \subseteq \check{D}$ be 
defined by $\widetilde{\Xi}_\bH = \bH^{\bG^\partial} \cdot \check{D}^\partial	
= \bG^\partial \cdot \bH \cdot \check{D}^\partial$.   This differential constructible 
set is defined by $\widetilde{\Xi}_{\bH,m} := \bG^0_m \cdot \cA_m \bH \cdot \check{D}^0_m \subseteq \cA_m \check{D}$ 
for $m$ large enough.  Letting $\underline{\widetilde{\Xi}_{\bH,m}}$ be the image of 
$\widetilde{\Xi}_{\bH,m}$ in the Schwarzian variety, we see that $\widetilde{\Xi}_\bH$ is 
defined by the differential equation $\sS_{\bG,\check{D}}(x) \in \underline{\widetilde{\Xi}_{\bH,m}}$.  
Let $\Xi_H$ be defined by $\chi(x) \in \underline{\widetilde{\Xi}_{\bH,m}}$.  If $a \in \Xi_\bH(\cM)$, 
let $\widetilde{a} = q^{-1} f(a)$ for any choice of a branch of $q^{-1}$.  Then 
$\chi(a) = \sS_{\bG,\check{D}}(\widetilde{a})$ so that $\widetilde{a} \in \widetilde{\Xi}_\bH$.  
Thus, there is some $g \in \bG(\CC)$, $b \in \check{D}(\CC)$ and $h \in \bH(\cM)$ with 
$\widetilde{a} = h^g \cdot b$.   Conversely, if $a \in S(\cM)$ and $f(a)$ lifts 
to some $\widetilde{a} \in \check{D}(\cM)$ for which there are $g \in \bG(\CC)$, 
$b \in \check{D}(\CC)$, and $h \in \bH(\cM)$ with $\widetilde{a} = h^g \cdot b$, then 
$\chi(a) = \sS_{\bG,\check{D}}(\widetilde{a}) \in \underline{\widetilde{\Xi}_{\bH,m}}$ so that 
$a \in \Xi_\bH(\cM)$ as claimed.
\end{proof}

\begin{Rk}
The differential constructible sets $\Xi_\bH$ and 
$\widetilde{\Xi}_\bH$ are not closed in general.
For each $d \leq \dim \bB$, consider 
$\widetilde{\Xi}_{\bH,d}$ defined by 
$$\widetilde{\Xi}_{\bH,d}(\cM) := 
\{ a \in \widetilde{\Xi}_{\bH}(\cM) ~:~ 
\dim \operatorname{Stab}_{\bG(\CC)} (a) \geq d \} \text{ .}$$ 
We let $\Xi_{\bH,d}$ be defined by 
$$\Xi_{\bH,d}(\cM) := \{ a \in \Xi_\bH(\cM) ~:~
(\exists \widetilde{a} \in \widetilde{\Xi}_{\bH,d}(\cM)
~:~ q(\widetilde{a}) = f(a) \} \text{ .}$$   Then 
$\Xi_{\bH,d}$ is differential algebraic and 
is closed in $X \smallsetminus \Xi_{\bH,d+1}$
(where we set $\Xi_{\bH,\dim \bB + 1} = 
\varnothing$).   We note that $\widetilde{\Xi}_{\bH,\dim \bB} = \check{D}^\partial$.  
\end{Rk}

We can identify the semiconstant points
using the differential constructible sets $\Xi_\bH$.
If we permit ourselves to regard the trivial group 
as a semisimple $\QQ$-algebraic group, then 
we may see $S^\partial$ as $\Xi_{\{1\}}$.   More
generally, if $S' \subseteq S_{\Gamma,G,M}$ is a 
semiconstant weakly special variety, then 
there are connected semisimple $\QQ$-algebraic 
groups $\bH_1$ and $\bH_2$ with $\dim \bH_1 > 0$,
a map of algebraic groups $\iota:\bH_1 \times 
\bH_2 \hookrightarrow \bG$ with finite kernel and a point $a \in D$ with 
$S' = q(\iota(\{ 1 \} \times \bH_1)(\RR)^+ \cdot a)$.  
If we let $\bH' := \iota( \{ 1 \} \times \bH_2 )$, 
then we see that every analytic quotient space of the form 
$q((\bH')^g(\RR)^+ \cdot b)$ with $g \in \bG(\RR)$ and 
$b \in D$ is semiconstant.  As above, we see that there
is a finite set ${\mathcal{S}}{\mathcal{C}}$ of 
connected semisimple $\QQ$-algebraic subgroups of $\bG$ so 
that every such $\bH'$ is $\bG(\RR)$ conjugate to some 
element of ${\mathcal{S}}{\mathcal{C}}$.   For 
$\bH \in \mathcal{H}$, we define $\Xi^\ft_\bH 
:= \Xi_\bH \smallsetminus \bigcup_{\bH' \in \mathcal{S}\mathcal{C}} 
\Xi_{\bH'}$.   Then we see that $\Xi_\bH^\ft(\cM) = \Xi_\bH(\cM)^\ft$ 
as the notation suggests.
If we define $\Xi^{[\ell]}{}^\ft := \bigcup_{\bH \in \cH^{[\ell]}} 
\Xi_\bH^\ft$, then we see that $\Xi^{[\ell]}{}^\ft(\cM) = \Xi^{[\ell]}(\cM)^\ft$

While in Lemma~\ref{H-orbit-eq} we required merely that 
$f:S \to S_{\Gamma,G,M}$ be bi-algebraic, to analyze the differential
varieties $\Xi_H$ we need a form of the Ax-Schanuel conjecture to 
hold.

\begin{Def}
\label{def:AS}
Let $f:S \to S_{\Gamma,G,M}$ be a bi-algebraic map to the 
complex quotient space $S_{\Gamma,G,M}$. We say that
$f:S \to S_{\Gamma,G,M}$ satisfies the Ax-Schanuel condition 
if whenever $a \in S(\cM)$ and $\widetilde{a} \in D(\cM)$ satisfy
$f(a) = q(\widetilde{a})$, we have $\trdeg_\CC \CC(a,\widetilde{a}) 
\geq \dim D + 1$ or $f(a)$ lies on a weakly special subvariety of $S_{\Gamma,G,M}$.  
\end{Def}

It is known that the period mapping associated to a polarized variation 
of Hodge structures is bi-algebraic~\cite[Theorem 1.3]{BKT} and 
satisfies the Ax-Schanuel condition~\cite[Theorem 1.1]{BaTs}.   In fact, 
the period mapping for variations of 
mixed Hodge structures satisfies the Ax-Schanuel
condition~\cite[Theorem 1.2]{Chiu21} or~\cite[Theorem 1.1]{GK}. 

We prove now a qualitative theorem towards Zilber-Pink on unlikely intersections
for the differential equations satisfied 
by the special varieties.

\begin{thm}
\label{thm:ZP-H}
Let $f:S \to S_{\Gamma,G,M}$ be a bi-algebraic map 
for which each Cartesian power $f^{\times N}:S^{\times N} 
\to S_{\Gamma,G,M}^{\times N}$
satisfies the Ax-Schanuel condition. Let 
$\mathbf{H} \leq \mathbf{G}$ be a nontrivial connected algebraic 
subgroup.  Let $\ell := \dim_\CC H(\RR) \cdot a$ for 
some (equivalently, any) $a \in D$.  
Let $X \subseteq S_\cM$ be an irreducible 
subvariety of the base change of $S$ to $\cM$ for 
which $f(X)$ is not contained in any proper weakly
special variety.  We suppose
that $\ell + \dim(X) < \dim \check{D}$.  
Then $\Xi_\bH^\ft \cap X$ is not 
Zariski dense in $X$.  In fact, there is a finite 
set $\mathcal{E}$ of proper $\cM$-weakly special
subvarieties of $S$ so that $\Xi_\bH^\ft \cap X
\subseteq X \cap \bigcup_{\widetilde{S} \in \mathcal{E}} f^{-1} \widetilde{S} \subsetneq X$.   
\end{thm}

\begin{proof}
Since the Kolchin topology is Noetherian, to find the set 
$\mathcal{E}$, it suffices to produce for each component 
$Z$ of the differential constructible set $\Xi_\bH \cap X$ 
some proper $\cM$-weakly special $\widetilde{S}$ with $Z \subseteq f^{-1} \widetilde{S}$. 
Let $Z$ be such a component of $\Xi_\bH \cap X$.  

Let $L$ be a finitely generated over $\CC$ subfield
of $\cM$ over which $X$ and $Z$ are defined.  
Let $N := \trdeg_\CC L + 1$.  Let $(a_i)_{i=1}^N$ be 
a Morley sequence in $Z(\cM)$ over $L$.   That is, 
$(a_1, \ldots, a_N) \in Z^{\times N}(\cM)$ and for every proper
differential subvariety $W \subsetneq Z^{\times n}$ defined over 
$L$, $(a_1, \ldots, a_N) \notin W(\cM)$.    (That such a sequence 
may be found in $Z^{\times N}(\cM)$ uses the fact that every point in $Z$ is 
fully transcendental.)  
Let $(\widetilde{a}_i)_{i=1}^N$ be a sequence of elements 
of $\check{D}(\cM)$ with $q(\widetilde{a}_i) = f(a_i)$.  
We compute an upper bound on $\trdeg_\CC(\widetilde{a}_1, 
\ldots, \widetilde{a}_N,a_1, \ldots, a_N)$. 

\begin{eqnarray*}
\trdeg_\CC \CC(\widetilde{a}_1, \ldots, \widetilde{a}_N,a_1, \ldots, a_N) & \leq &
	\trdeg_\CC L (\widetilde{a}_1, 
\ldots, \widetilde{a}_N,a_1, \ldots, a_N)  \\
 & = & \trdeg_\CC L + \trdeg_L 
 L(\widetilde{a}_1, \ldots, \widetilde{a}_N,a_1, \ldots, a_N)	\\
 & \leq & \trdeg_\CC L + \trdeg_L L (\widetilde{a}_1, \ldots, \widetilde{a}_N)  + 
 \trdeg_L L(a_1, \ldots, a_N) \\
 & \leq & N + N \ell + N \dim(X) \\
 & \leq &  N \dim (\check{D})  
\end{eqnarray*}

By the Ax-Schanuel condition, this is 
only possible if there is a proper 
special subvariety 
$S' \subseteq S_{\Gamma,G,M}^{\times N}$ with 
$(a_1, \ldots, a_N) 
\in (f^{\times N}){}^{-1}(S')(\cM)$.  
Let $j$ be minimal so that if $\pi_j: S_{\Gamma,G,M}^{\times N} 
\to S_{\Gamma,G,M}^{\times j}$  is the projection onto the 
first $j$ coordinates, then $\pi_j(S') 
\neq S_{\Gamma,G,M}^{\times j}$. The point $a_j$ is a 
generic point of $Z$ over the differential field generated by 
$L(a_1, \ldots, a_{j-1})$, but it also belongs to the 
$\cM$-weakly special variety $\widetilde{S} := f^{-1}(\rho( (\{ (f(a_1), \ldots, f(a_{j-1})) \} \times S_{\Gamma,G,M}) \cap S' ))$, where 
$\rho:S_{\Gamma,G,M}^{\times j} \to S_{\Gamma,G,M}$ is the projection to the last coordinate.     Thus, $Z \subseteq \widetilde{S}$.

\end{proof}

\begin{Rk}
If the map $f:S \to S_{\Gamma,G,M}$ satisfies 
a suitable Ax-Schanuel with derivatives theorem, 
as holds, for example, for Shimura varieties~\cite[Theorem 12.3]{MPT},
then  Theorem~\ref{thm:ZP-H} may be strengthened to a statement 
in which $X$ may be taken to be a differential variety. We will return to 
this point in Section~\ref{trivial-types}.	
\end{Rk}

\begin{Rk}
\label{rk:ft-to-nonconstant}
Working by induction on the dimension of $X$, we may upgrade Theorem~\ref{thm:ZP-H} to the assertion that $\Xi_\bH \cap (X(\cM) \smallsetminus X(\CC))$ 
is not Zariski dense in $X$.   However, to include the semiconstant points, we may be obliged to extend $\mathcal{E}$ to have a family of $\cM$-weakly 
special varieties parameterized by a the $\CC$-points of some algebraic variety.	  We spell out the details of this strengthening with Theorem~\ref{thm:constant-to-qual}.
\end{Rk}
 
The next lemma will permit us to capture all special varieties of a fixed dimension 
by finitely many differential varieties of the form $\Xi_\bH$.
 
\begin{lem}
\label{special-ODE}
Let $f:S \to S_{\Gamma,G,M}$ be bi-algebraic with 
$S_{\Gamma,G,M}$ an arithmetic quotient.  
For each natural number $\ell \leq \dim D$, 
there is a finite set $\cH^{[\ell]}$ of 
semisimple $\QQ$-algebraic subgroups of $\bG$ so that 
for each $\bH \in \cH^{[\ell]}$ there is some 
$a \in D$ with $q(\bH(\RR)^+ \cdot a) 
\subseteq S_{\Gamma,G,M}$ being a
special variety of dimension $\ell$
and 
$\cS_S^{[\ell]}(\cM) \subseteq \bigcup_{\bH \in 
\cH} \Xi_\bH(\cM)$.
\end{lem}

\begin{proof}
As is well known 
(see, for instance,~\cite[Proposition 12.1]{Ri}), 
there is finite set $\cH$ of connected 
semisimple $\QQ$-algebraic
subgroups of $\bG$ so that for any connected 
semisimple $\QQ$-algebraic subgroup $\widetilde{\bH} \leq \bG$ 
of $\bG$ there is some $\bH \in \cH$ and $g \in 
\bG(\RR)$ with $\widetilde{\bH} = \bH^g$.  Let 
$$\cH^{[\ell]} := \{ \bH \in \cH ~:~ (\exists 
a \in D, g \in \bG(\RR)) ~q(\bH^g(\RR)^+ \cdot a) \subseteq 
S_{\Gamma,G,M} $$ $$ \text{ is a special subvariety of dimension } \ell \} 
\text{ .}$$
If $a \in \cS_S^{[\ell]}(\cM)$, then there is some 
special subvariety $S' \subseteq S_{\Gamma,G,M}$ of
dimension $\ell$ so that $f(a) \in S'$.  Express
$S'$ as $q(\bH'(\RR)^+ \cdot a)$ for some $a \in D$ and
semisimple $\QQ$-algebraic subgroup $\bH' \leq \bG$.
We then find $\bH \in \cH$ and $g \in \bG(\RR)^+$
so that $\bH' = \bH^g$, giving that $\bH \in 
\cH^{[\ell]}$.  By Lemma~\ref{H-orbit-eq}, 
$a \in \Xi_H(\cM)$, as claimed.
\end{proof}

An effective Zilber-Pink theorem may be 
deduced from Theorem~\ref{thm:ZP-H}.   

\begin{cor}
\label{effective-ZP}
Let $f:S \to S_{\Gamma,G,M}$ be bi-algebraic with 
$S_{\Gamma,G,M}$ an arithmetic quotient. 
We suppose that $S$ is given with a fixed quasi-
projective embedding.   
Then there is a constant $C$ so that for any 
natural number $\ell$ and  
any irreducible subvariety $X \subseteq S_\cM$
with $\dim(X) + \ell < \dim(S)$,
there is a proper subvariety $Y \subsetneq X$ with   
$(X(\cM) \smallsetminus X(\CC)) \cap \cS_S^{[\ell]} \subseteq Y$ and
$\deg(Y) \leq C \deg(X)^{\dim(S)}$. 
\end{cor}
\begin{proof} 
Let $\cH^{[\ell]}$ be given by Lemma~\ref{special-ODE}.
From that lemma, we see that $\cS_S^{[\ell]} \subseteq 
\bigcup_{\bH \in \cH^{[\ell]}} \Xi_\bH := \Xi^{[\ell]}$.     
By Theorem~\ref{thm:ZP-H} and Remark~\ref{rk:ft-to-nonconstant}, for each $\bH \in \cH^{[\ell]}$, 
the Zariski closure of $(X(\cM) \smallsetminus X(\CC)) \cap \Xi_\bH$
is a proper subvariety of $X$.  Hence, 
the Zariski closure $Y$ of $X \cap (\Xi^{[\ell]} 
\smallsetminus X^\partial)$ is a proper
subvariety of $X$ and contains $(X(\cM) \smallsetminus X(\CC)) \cap \cS_S^{[\ell]}$. 
By~\cite[Corollary 11]{Bin},
the degree of this Zariski closure $Y$ is bounded by
$C \deg(X)^{\dim(S)}$ where $C$ depends on $S$ and $f$, but
not on $X$. 
\end{proof}

\begin{Rk}
The constant $C$ appearing in 
Corollary~\ref{effective-ZP} may be computed from bounds on the 
degrees of the differential equations defining $\Xi^{[\ell]}$. 	
\end{Rk}

One might ask how far the $Y$ of Corollary~\ref{effective-ZP} is from 
the Zariski closure of $(X(\cM) \smallsetminus X(\CC)) \cap \cS_S^{[\ell]}$. 
The following conjecture implies that they are in 
fact equal.

\begin{conj}
\label{special-density}
Suppose that $f:S \to S_{\Gamma,G,M}$ is 
bi-algebraic, $X \subseteq S_\cM$ is an 
irreducible algebraic subvariety of the 
base change of $S$ to $\cM$, $\dim X > 0$, 
$\ell \in \ZZ_+$, 
and $X(\cM) \cap \Xi^{[\ell]}(\cM)^\ft$ is 
Zariski dense in $X$.  We assume moreover 
that $X$ is not contained in any $\cM$-weakly special
subvariety of $S$.    
Then $X(\cM) 
\cap \cS_S^{[\ell]}(\cM)^\ft$ is Zariski dense 
in $X$. 	
\end{conj}

Conjecture~\ref{special-density} may 
be understood as a ``likely 
intersections'' counterpart to
the Zilber-Pink conjecture.  
Variants have been studied
by Klingler and Otwinowska
in~\cite{KO}.

\section{Another approach to function field Zilber-Pink}
\label{sec:qualitative}

In this section we explain how a uniform version of the Zilber-Pink conjecture 
may be deduced from a weak version in which only varieties defined over $\CC$ are considered.

We start by specifying what we would mean by a weak Zilber-Pink conjecture.

\begin{Def}
\label{def:weak-ZP}
We say that weak Zilber-Pink holds for the bi-algebraic $f:S \to S_{\Gamma,G,M}$
if whenever $X \subseteq S$ is 
an irreducible complex algebraic subvariety of $S$  
for which $f(X)$ is not contained in a proper special subvariety of $S_{\Gamma,G,M}$, 
then the union of all strongly atypical components of intersections of $X$ with 
pullbacks of strongly special subvarieties of $S_{\Gamma,G,M}$ is not Zariski dense in $X$.	 
Here a component $U$ of $X \cap f^{-1}(\widetilde{S})$ is strongly atypical if 
$\dim(U) > \max \{ \dim(S_{\Gamma,G,M}) - (\dim(X) + \dim(\widetilde{S})), 0 \}$. 
\end{Def}

\begin{Rk}
The weak Zilber-Pink condition is weak in two senses:  we consider only subvarieties $X$ defined 
over $\CC$ and we make an assertion only about atypical components of dimension at least one.  It is strong in the sense that 
it is an assertion about atypical intersections and not merely unlikely intersections.  We discuss the apparent gap between 
atypical and unlikely intersections at the end of this section.  	
\end{Rk}

Our main result is that if weak Zilber-Pink holds for $f:S \to S_{\Gamma,G,M}$, then the function field version of Zilber-Pink for unlikely
intersections holds.
 
\begin{thm}
\label{thm:constant-to-qual}  If weak Zilber-Pink holds for $f:S \to S_{\Gamma,G,M}$, $X \subseteq S_\cM$ is an
irreducible algebraic subvariety of the base change of $S$ to $\cM$ for which $f(X)$ is not contained in any proper
special subvariety, then $(X(\cM) \smallsetminus X(\CC)) \cap \cS_S^{[\leq \ell]}(\cM)$ is not Zariski dense in $X$ where 
$\ell = \dim S_{\Gamma,G,M} - (\dim(X) + 1)$. 	
\end{thm}

\begin{proof}
Let $Z$ be the $\CC$-Zariski closure of $X$.  That is, $Z$ is the smallest subvariety of $S$ defined over 
$\CC$ with $X \subseteq Z_\cM$.   If $Z_\cM = X$, then weak Zilber-Pink already says that the conclusion 
we desire holds for $X$.   On the other hand, if $Z = S$, then~\cite[Theorem 1.2]{CGMM} implies that 
$X(\cM)^\ft \cap \cS_S^{[\leq \ell]}(\cM)$ is not Zariski dense in $X$.  Indeed,~\cite[Theorem 1.2]{CGMM} is stated 
with $\mathbb{A}^n$ as the ambient variety, but the proof applied mutatis mutandis for any given ambient 
variety.  In the notation of~\cite[Theorem 1.2]{CGMM}, we have $k = \mathbb{C}$ and $V = X$ and have replaced
$\mathbb{A}^n$ by $S$.    The result follows as every 
special subvariety of $S$ is defined over $\CC$,  so that 
$X(\cM)^\ft \cap \cS_S^{[\leq \ell]}(\cM) \subseteq X(\cM)^\ft \cap \bigcup_{ \tiny \begin{array}{c} Y \subseteq X \\
\CC-\text{algebraic subvariety} \\ \dim(Y) \leq \ell \end{array} } Y(\cM)$, which is not Zariski dense in $X$ by~\cite[Theorem 1.2]{CGMM}.

For the remainder of the proof we consider the case that $X \subsetneq Z_\cM \subsetneq S_\cM$.

We define two sets of irreducible varieties.

$$\mathcal{A} := \{ U \subseteq Z ~:~ U \text{ is a component of an intersection } Z \cap f^{-1} \widetilde{S} $$
$$ \text{ where } \widetilde{S} \text{ is strongly special with } \dim \widetilde{S} \leq \ell \text{ and } 
\dim(U) > \dim(Z) - (\dim X + 1) \}  $$ 

and 

$$\mathcal{T} := \{ U \subseteq Z ~:~ U \text{ is a component of an intersection } Z \cap f^{-1} \widetilde{S} $$
$$ \text{ where } \widetilde{S} \text{ is strongly special with } \dim \widetilde{S} \leq \ell \text{ and } 
\dim(U) \leq \dim Z - (\dim X  + 1) \}  $$

Observe that each $U \in \mathcal{A}$ is actually strongly atypical. Note that because $f(X)$ is not contained in 
a proper special variety, neither is $Z$.  Thus, by weak Zilber-Pink, $\bigcup \cA$ is not Zariski dense in 
$Z$.  Thus, $\overline{\bigcup \cA}$, being a proper subvariety of $Z$ defined over 
$\CC$, does not contain $X$.  Hence, $\bigcup A \cap X  \subseteq \overline{\bigcup A} \cap X$ is not 
Zariski dense in $X$.

On the other hand, each $U \in \mathcal{T}$ is a $\CC$-variety of dimension strictly less than the codimension 
of $X$ in $Z$.  By~\cite[Theorem 1.2]{CGMM}, $X \cap \bigcup \mathcal{T}$ is not Zariski dense in $X$.

For any special variety $\widetilde{S}$ with $\dim(S) \leq \ell$, we have $X \cap f^{-1} \widetilde{S} = 
X \cap (Z \cap f^{-1} \widetilde{S}) \subseteq X \cap (\bigcup \mathcal{A} \cup \bigcup \mathcal{T})$.  
Thus, $X \cap \cS_S^{[\leq \ell]} \subseteq \overline{X \cap \bigcup \mathcal{A}} \cup \overline{X \cap \bigcup{T}}$
which is not Zariski dense in $X$. 
\end{proof}

\begin{Rk}
 Note that our ostensibly stronger conclusion at the end of the proof of Theorem~\ref{thm:constant-to-qual}
 that $X \cap \cS_S^{[\leq \ell]}$ is not Zariski dense in $X$, rather than just that $(X(\cM) \smallsetminus X(\CC)) \cap \cS_S^{[\leq \ell]}$
 is not Zariski dense in $X$, is not really a strengthening as in the subcase under consideration $X(\CC)$ is not 
 Zariski dense in $X$ because $X$ did not descend to $\CC$.  	
\end{Rk}

\begin{Rk}
The quality of our conclusion in Theorem~\ref{thm:constant-to-qual} that the nonconstant unlikely intersections are not Zariski dense in 
$X$ appears to be weaker than what appears in the weak Zilber-Pink statement which is about atypical intersections.  In fact, in general, 
Zilber-Pink expressed in terms of atypical intersections is equivalent to Zilber-Pink in Pink's formulation which is expressed in terms
of unlikely intersections~\cite[Section 12]{BD21}.  The deduction
of Zilber-Pink  for atypical intersections from Zilber-Pink for unlikely intersections in \cite{BD21} uses the Ax-Schanuel property, which always holds in the required settings (see the references above below 3.3). 
It is part of the reduction effected there of ZP to finiteness
of ``optimal points''. 
Alternatively, the deduction can be made  by systematically intersecting with very general linear spaces.  All zero-dimensional atypical intersection components must be unlikely. Consider  
atypical components $A\subset X$ of some higher dimension $d$.
There are at most countably many as special subvarieties form
a countable collection. A very general linear subvariety
of codimension $d$ will intersect $X$ and all $A$ in the expected dimensions, the intersections will be distinct and unlikely.
\end{Rk}

\section{Trivial minimal types in differentially closed 
fields}
\label{trivial-types}

In~\cite{FrSc}, the Ax-Lindemann-Weierstra{\ss} with derivatives theorem of~\cite{Pi13} 
is interpreted to say that certain definable sets relative to the theory 
of differentially closed fields of characteristic zero are strongly minimal, have trivial forking geometry,
and have non-$\aleph_0$-categorical induced structure.   Up to this point in this paper 
we have used only the algebraic form of the Ax-Schanuel condition.
The main theorem of~\cite{MPT} gives functional transcendence statements
for algebraic differential equations for uniformizing maps of Shimura 
varieties generalizing the results for the $j$-function.  As such, we 
identify associated strongly minimal sets with forking geometry analogous to 
that of the differential equations for the $j$-function.  As with the 
results of~\cite{FrSc}, we leverage this interpretation to prove further functional 
transcendence theorems.  

Much more general Ax-Schanuel theorems are announced by 
Bl\'{a}zquez Sanz, Casale, Freitag, and Nagloo 
in~\cite{BSCFN}.  Analogous results on the model theoretic 
properties of the associated differential equations follow in 
each of the cases they consider.  In~\cite{BSCFN}, strong minimality and forking triviality 
for the differential equations associated to the covering maps of
simple Shimura varieties are established and (non-)orthogonality is 
described geometrically in much the same way as is done here (which is 
not surprising as our methods and theirs follow the analysis of~\cite{FrSc}).
A subtlety here is that we consider as well the case where the 
underlying Shimura variety is not simple, observing that there can be 
a real distinction between minimality and strong minimality related to 
the notion of $\delta$-Hodge genericity.

In this section, we use freely the ideas of 
geometric stability theory. 
See~\cite{GST} for details on such topics 
as U (also called ``Lascar'') rank, multiplicity, 
orthogonality, and Morley 
sequences. 

\begin{Def}
Let $f:S \to S_{\Gamma,G,M}$ be bi-algebraic and let 
$K$ be a differential field with field of constants $\mathbb{C}$. 
For any point $\bar{a} \in  (\mathbf{G}^\partial \backslash \check{D})(K)$ 
in the associated Schwarzian variety, $X_{S,\bar{a}}$ is the 
differential subvariety of $S$ defined by $\chi_S(x) = \bar{a}$.  	
\end{Def}

\begin{Rk}
At this level of generality, we cannot say much about $X_{S,\bar{a}}$.   
If $f$ is not surjective, then it may happen that $X_{S,\bar{a}} = \varnothing$.
As such, we usually insist that $\bar{a}$ belongs to the differentially 
constructible set obtained as the image of $S$ under $\chi_S$.   	
\end{Rk}

\begin{prop}
\label{not-aleph-nought}
If $S = S_{\Gamma,G,M}$ is a Shimura variety and 
$\bar{a}$ belongs to the image of $\chi_S$, 
then $X_{S,\bar{a}}$ does not have $\aleph_0$-
categorical	induced structure.
\end{prop} 
\begin{proof}
For any $\gamma \in G$ in the commensurator of 
$\Gamma$, $\Gamma^\text{comm}$,
which is all of $\bG(\mathbb{Q})^+$, 
the analytic variety 
$T_\gamma := \{ (\pi(\tau),\pi(\gamma \tau) ) ~:~
\tau \in D \}$ is an algebraic subvariety of 
$S \times S$ which restricts to a finite-to-finite
correspondence on $X_{S,\bar{a}}$ and the 
set of distinct such correspond	to 
the infinite coset space $\Gamma^\text{comm}/\Gamma$.
Thus, there are infinitely
many distinct $0$-definable
subsets of $X_{S,\bar{a}}^2$ 
so that its induced structure
is not $\aleph_0$-categorical.
\end{proof}

Some interesting subtleties emerge in the study of the differential
varieties $X_{S,\bar{a}}$ for general Shimura varieties not seen for 
case of the $j$-line.  It may happen that a point $a \in S(\cM)$ 
is Hodge generic, in the sense that it does not lie on any proper 
special subvariety, but the differential variety $X_{S,\bar{a}}$ is  
equal to $X_{S,\bar{b}}$ for some $b$ which is \emph{not} Hodge generic.  
Often when this happens, the Lascar and Morley ranks of $X_{S,\bar{a}}$ will
disagree. Such equations appear implicitly in~\cite{HrSc} with the 
differential variety $F_2$ of~\cite[Corollary 2.7]{HrSc}.   

With the next proposition we show that the \emph{type}
 of a generic point in such a $X_{S,\bar{a}}$ for $S$
irreducible is minimal.  We address the question of 
strong minimality afterwards.

\begin{prop}
\label{minimaltype}
Let $S$ be a connected, irreducible Shimura variety.  Express
$S$ as $S = S_{\Gamma,G,M}$.   Let $K \subseteq L$ be an
extension of differential fields each  
with field of constants $\mathbb{C}$ and $a \in S(L)
\smallsetminus S(K^\text{alg})$ an $L$-valued 
point of $S$ which is not algebraic over
$K$ having $\bar{a} := \chi_S(a) \in 
(\mathbf{G}^\partial \backslash \check{D})(K)$.  
Then $\tp(a/K)$ is minimal.
\end{prop}

Before we commence with the proof, 
let us dispense with some niceties.  First, for us a Shimura variety
is positive dimensional.  Secondly, by ``irreducible'' we mean that
the Hermitian domain $D$ is irreducible.  From the point of view of the 
Shimura variety itself, this means that we cannot find Shimura varieties
$S_1$ and $S_2$ and a finite, dominant map $S_1 \times S_2 \to S$
of Shimura varieties, that is, as quotient spaces. 

\begin{proof}
Using the Seidenberg embedding theorem and shrinking $L$ to be a finitely generated differential field 
if need be, 
we may regard $L$ as differential subfield of $\cM$.  Let $\widetilde{a} 
\in \check{D}(\cM)$ so that that $a = \pi(\widetilde{a})$. 

Since $a$ is not algebraic over $K$, $U(a/K) \geq 1$.  
We check now that $U(a/K) \leq 1$.  That is, for 
any differential field $M$ containing $K$  
either $a \in S(M^\text{alg})$ or 
$a$ is independent from $M$ over $K$.  As above, 
we may take $M$ to be a finitely generated over $K$
differential subfield of $\cM$.   

We suppose that $a \notin S(M^\text{alg})$. By~\cite[Proposition 4.2]{Sc-ade}, 
the transcendence degree over $K$
of the differential field generated over $K$ by $a$ is at most
$\dim \mathbf{G}$.  With the following calculation we will show 
that, in fact, the transcendence degree over $L$ 
of the differential field generated over $M$ by $a$ is exactly 
$\dim(\mathbf{G})$.   

Let $(a_i)_{i=0}^\infty$ be a Morley sequence in 
$\operatorname{tp}(a/M)$ with $a_0 = a$.   Let $g_i \in G$
so that $a_i = \pi(g_i \widetilde{a})$. Let $r \geq \dim \mathbf{G}$ 
and set $b_i := \nabla_r(a_i)$.  Let $M' \subseteq M$  
be a finitely generated (over $M$) subfield over 
which $\bar{a}$ and the algebraic locus of $c_0$ over $M$ are 
defined.    Let $N \in \mathbb{N}$ be a natural number.   

We assume towards a contradiction that
$\vec{a} := (a_0, \ldots, a_{N-1}) \in S^N(\cM)$ is Hodge generic
and that $\trdeg_M M \langle a \rangle < \dim \mathbf{G}$.  
For the first step of the following 
computation we use~\cite[Theorem 12.3]{MPT}.
 
\begin{eqnarray*}
1 + N \dim \mathbf{G} & \leq & \trdeg_\mathbb{C} \mathbb{C} (g_0 \tilde{a}, 
\ldots, g_{N-1} \tilde{a}, c_0, \ldots, c_{N-1}) \\
	& = & \trdeg_\mathbb{C} \mathbb{C}(\widetilde{a}, (c_i)_{i=0}^{N-1}) \\
	& \leq & \trdeg_\mathbb{C} M'(\widetilde{a}, (c_i)_{i=0}^{N-1}) 	\\
	& \leq & \trdeg_\mathbb{C} M'(\widetilde{a}) +  N \trdeg_{M'} M'(c_0) \\
	& \leq & \trdeg_\mathbb{C} M'(a) + N (\dim \mathbf{G} - 1)
\end{eqnarray*}
   
If we take $N  \geq \trdeg_\mathbb{C} M'(a)$, then this inequality fails.  Thus,
our hypothesis that $\trdeg_M M \langle a \rangle 
< \dim \mathbf{G}$ and that 
$(a_i)_{i=0}^{N-1}$ is 
Hodge generic must be wrong.  We know that $a$ depends 
on $M$ over $K$.  Thus, $\vec{a}$ is not Hodge generic.  
Let $S' \subseteq S^{\times N}$ be a proper
strongly special subvariety with $(a_0,\ldots,a_{N-1}) 
\in S'(\cM)$.  Since $a_0$ is Hodge generic in $S$ and 
$(a_i)_{i=0}^\infty$ is $M$-indiscernible, each 
$a_i$ is Hodge generic in $S$.  Thus, $S'$ projects
dominantly on each factor and the variety
$V := \pi((\{(a_0,\ldots,a_{N-2})\} \times S) \cap S')$ 
(where $\pi:S^{\times N} \to S$ is the projection to the 
last coordinate) is a proper $\cM$-weakly special subvariety of 
$S$ containing $a_{N-1}$.  

Using irreducibility of $S$ as a Shimura variety we 
see that $V$ must be finite.  Indeed, if $V$ were infinite, 
then it could be expressed as $\rho ( \{ b \} \times S_2)$ 
where $\rho:S_1 \times S_2 \to S$ is a finite map of Shimura
varieties with $S_1$ and $S_2$ infinite and $b$ an $\cM$-valued 
point.  Since $S$ is irreducible, the image if $\rho$ is not all of 
$S$.  Thus, $a_{N-1}$ belongs to the proper
special variety $\rho(S_1 \times S_2)$, which implies
by indiscernibility that $a$ does, too, contradicting 
its Hodge genericity.  Thus, $V$ is finite.

Because $(a_i)_{i=0}^\infty$ is indiscernible, 
$(a_0,\ldots,a_{N-2},a_j) \in S'(\cM)$ for 
all $j \geq N-1$. It follows 
by the the pigeonhole principle that there are 
$i > j \geq N$ with $a_i = a_j$.  By indiscernibility
again, $a_i = a_j$ for all $i$ and $j$.  Since 
$(a_i)_{i=0}^\infty$ is a Morley sequence, and, 
in particular, is independent, this can only 
happen if $a_i \in S(L^\text{alg})$.    
\end{proof}

The failure of strong minimality comes from 
differential equations associated with proper special 
subvarieties.  We isolate the relevant condition 
with the next definition.

\begin{Def}
Let $f:S \to S_{\Gamma,G,M}$ be bi-algebraic and $a \in S(\cM)$.  
We define $\delta$-$\mathbf{M}\mathbf{T}(a)$ to be the 
semi-simple $\mathbb{Q}$-algebraic group $\mathbf{H} \leq \mathbf{G}$
if $a \in \Xi_{S,\mathbf{H}}(\cM)$ but for all proper 
semi-simple $\mathbb{Q}$-algebraic subgroup $\mathbf{H}' < \mathbf{H}$,
$a \notin \Xi_{S,\mathbf{H}'}(\cM)$.  We say that 
$a$ is differentially Hodge generic if $\delta$-$\mathbf{M}\mathbf{T}(x) 
= \mathbf{G}$.   Note that $\delta$-$\mathbf{M}\mathbf{T}(x)$ is only 
well-defined up to $\mathbf{G}$ conjugacy.
\end{Def}

\begin{Rk}
Our use of the word ``differentially Hodge generic'' is inspired by 
Buium's work in~\cite{Bu} though we are not 
following precisely the same formalism here.   	
\end{Rk}

\begin{prop}
Let $S$ be a connected, irreducible Shimura variety.  Express
$S$ as $S = S_{\Gamma,G,M}$.   Let $K \subseteq L$ be an
extension of differential fields each with  
with field of constants $\mathbb{C}$ and $a \in S(L)
\smallsetminus S(K^\text{alg})$ an $L$-valued 
point of $S$ which is not algebraic over
$K$ having $\bar{a} := \chi_S(a) \in 
(\mathbf{G}^\partial \backslash \check{D})(K)$.  
If $a$ Hodge generic but is \emph{not} $\delta$-Hodge generic, then 
$RM(a/K) > 1$.
\end{prop}

\begin{proof}
Without loss of generality, we may 
assume that $K \subseteq L$ are 
finitely generated as
$\mathbb{C}$-differential algebras,  
and that $L \subseteq \cM$.
We will check that for each proper Kolchin 
closed subset $Z \subsetneq X_{S,\bar{a}}$, 
there is some $b \in (X_{S,\bar{a}} \smallsetminus
Z)(\cM)$ with $b$ not algebraic over $K$ and 
$\tp(b/K) \neq \tp(a/K)$.  Thus, the 
Cantor-Bendixson rank of $\tp(a/K)$ will be 
at least two, and, \emph{a fortiori}, 
$RM(a/K) \geq 2$.

Fix $\widetilde{a}$ with $a = \pi(\widetilde{a})$. 
Since $a \in \Xi_\mathbf{H}(\cM)$, there is 
some $g \in \mathbf{G}(\mathbb{C})$ with 
$g \widetilde{a} \in \mathbf{H} \cdot \check{D}^\partial$.  Multiplying by another element of
$G$ if need be, we have that 
$b := \pi (g \widetilde{a}) \in S'(\cM) := 
S_{\Gamma \cap H, H, M \cap H}(\cM)$ is an element of 
a proper special subvariety and 
$b \notin S'(K^\text{alg})$.   Indeed, for any 
$\gamma$ in the commensurator of $\Gamma$ 
(which under our hypotheses is just $\mathbf{G}(\mathbb{Q})^+$), 
$b_\gamma := \pi (\gamma g \tilde{a})$
belongs to the special variety $S_{\Gamma \cap H^\gamma,H^\gamma,M \cap H^\gamma}$ and is 
not algebraic over $K$.  Since $b_\gamma$ is not 
Hodge generic in $S$, $\tp(b_\gamma/K) \neq \tp(a/K)$.

It remains to check that there is no proper differential
subvariety $Z \subsetneq X_{S,\bar{a}}$ with 
$b_\gamma \in Z(\cM)$ for all such $\gamma$.  Consider
any differential regular function $h$ on $S$.  Then 
$h(\pi(y \cdot \widetilde{a})) = 0$ defines 
an analytic subvariety of $G$ where we regard $y$ as 
a variable ranging over $G$.   The commensurator group 
of $\Gamma$ is dense in the Euclidean topology in $G$.
Hence, this equation vanishes for all $g \in G$, 
implying that $h$ vanishes on all of $X_{S,\bar{a}}$.     
\end{proof}

On the other hand, for $\delta$-Hodge generic points
of irreducible Shimura varieties,
the types are strongly minimal.

\begin{prop}
Let $S$ be a connected, irreducible Shimura variety.  Express
$S$ as $S = S_{\Gamma,G,M}$.   Let $K \subseteq L$ be an
extension of differential fields each with  
with field of constants $\mathbb{C}$ and $a \in S(L)
\smallsetminus S(K^\text{alg})$ an $L$-valued 
point of $S$ which is $\delta$-Hodge generic 
and not algebraic over
$K$ having $\bar{a} := \chi_S(a) \in 
(\mathbf{G}^\partial \backslash \check{D})(K)$.  
Then $\tp(a/K)$ is strongly minimal.	
\end{prop}

\begin{proof}
We have already seen with Proposition~\ref{minimaltype}
that $\tp(a/K)$ is minimal.  It suffices to 
check that this type is isolated from all other 
nonalgebraic types.   Let $\cH$ be a finite 
set of proper, semisimple $\mathbb{Q}$-algebraic 
subgroups of $\mathbf{G}$ for which every 
such semisimple, $\mathbb{Q}$-algebraic subgroup
of $\mathbf{G}$ is $G$-conjugate to some element of 
$\cH$.  Let $Z := \bigcup_{\mathbf{H} \in \cH}
\Xi_\mathbf{H} \cap X_{S,\bar{a}}$.  Then 
$a \in X_{S,\bar{a}} \smallsetminus Z$ and we claim 
that if $a' \in (X_{S,\bar{a}} \smallsetminus Z)(\cM)$
is not algebraic over $K$, then $\tp(a/K) = \tp(a'/K)$.
Indeed, such an $a'$ is nonconstant (because every 
element of $X_{S,\bar{a}}$ is nonconstant) and 
is $\delta$-Hodge generic.  Thus, by 
Proposition~\ref{minimaltype}, 
$\trdeg_K K \langle a' \rangle = \dim \mathbf{G}$.  
It remains to check that $X_{S,\bar{a}}$ has 
only one generic component.  

Take $\widetilde{a}$ so that $a = \pi(\widetilde{a})$
and $g$ so that $a' = \pi(g \widetilde{a})$.   
Consider some algebraic differential equation
$H(x) = 0$ satisfied 
by $a'$.  Consider the analytic
equation $h(y \cdot \widetilde{a}) = 0$ with 
the variable $y$ ranging over $G$.  Because the 
locus of $a'$ has full dimension, this equation would
cut out an analytic subset of $G$ of full dimension. 
Since $G$ is irreducible, this equation would have
to vanish everywhere.  That is, this equation cannot 
distinguish $a'$ from $a$.  	
\end{proof}

With the next lemma we observe that 
all dependences between the types we have been 
considering may be explained by special varieties.
The proof reprises that of Proposition~\ref{minimaltype}.

\begin{lem}
\label{dependence-lem}
Let $S_1$ and $S_2$ be connected, 
irreducible, pure Shimura varieties.  
We express these as $S_i = S_{\Gamma_i,G_i,M_i}$ 
for $i = 1$ or $2$.  Let $K \subseteq L$ be an
extension of differential fields each with  
with field of constants $\mathbb{C}$ and 
$a_i \in S_i(L)$ for $i = 1$ or $2$ be Hodge 
generic points each of which is not algebraic 
over $K$.  Then $a_1$ and $a_2$ are 
dependent over $K$ if and only if there
is a special subvariety $T \subseteq S_1 \times S_2$
with $(a_1,a_2) \in T(L)$ and each projection 
$T \to S_i$ is finite and dominant. 	
\end{lem}
\begin{proof}
The right to left implication is immediate as the 
relation $T$ expresses $a_1$ and $a_2$
as being interalgebraic.  We focus on proving the 
left to right implication.

Replacing $K$ by a finitely generated differential 
$\mathbb{C}$-algebra
 over which $\tp(a_1,a_2/K)$ is defined and using 
the Seidenberg embedding theorem, we may assume that
$L = \cM$ and that $K$ is finitely generated.  
Swapping the roles of $a_1$ and $a_2$ if need be, 
we may assume that $\dim \bG_1 \geq \dim \bG_2$.

Write $\pi_i:D_i \to S_i$ for the covering map 
expressing $S_i$ as $S_{\Gamma_i,G_i,M_i}$ 
and fix some 
$\widetilde{a}_i$ with $\pi_i(\widetilde{a}_i) = 
a_i$ for $i = 1$ or $2$.  Let $r > 
 \dim \bG_1$ and set $b  := (a_1,a_2)$ and  
$c := \nabla_r(a_1,a_2)$.   Let $K_0 \subseteq K$ 
be a finitely generated $\mathbb{C}$-algebra
over which the algebraic locus of $(b,c)$ is defined.
Since $\tp(a_1/K)$ and $\tp(a_2/K)$ are minimal, 
we see that $K_0(b,c)$ is algebraic over 
$K_0(a_1,\nabla_r(a_1))$.  Thus, 
\begin{equation}
\label{trdeg-dep}
\trdeg_{K_0} K_0(b,c) = \trdeg_{K_0} K_0(\nabla_r(a_1))
= \dim \bG_1 
\end{equation}
by Proposition~\ref{minimaltype}. 

Let $(b_i,c_i)_{i=0}^\infty$ be a Morley sequence
in $\tp (b,c/K)$.   For each $i \in \mathbb{N}$, take
$(g_{1,i},g_{2,i}) \in G_1 \times G_2$ so that 
$(\pi_1(g_{1,i} \widetilde{a}_1), \pi_2(g_{2,i} \widetilde{a}_2)) = b_i$. 

Exactly as in the proof of Proposition~\ref{minimaltype}
using Equation~\ref{trdeg-dep} 
we compute that 
\begin{equation}
\label{Hg2-upper}
\trdeg_{\CC}\CC(\widetilde{a}_1,\widetilde{a}_2,
b_0, \ldots b_{N-1}, c_0, 
\ldots, c_{N-1}) \leq \trdeg_\CC K_0(\widetilde{a}_1,\widetilde{a}_2) + N \dim \bG_1  \text{ .}
\end{equation}

However, if $(a_1,a_2) \in S_1 \times S_2$ were 
Hodge generic, then $(b_0, \ldots, b_{N-1})$
would be Hodge generic in $(S_1 \times S_2)^N$
which would imply by the main theorem of~\cite{MPT}
that
\begin{equation}
\label{Hg2-lower}
1 + N \dim \bG_1 + N \dim \bG_2 
\leq \trdeg_{\CC}\CC(\widetilde{a}_1,\widetilde{a}_2,
b_0, \ldots b_{N-1}, c_0, 
\ldots, c_{N-1}) 
\end{equation}
 
Inequalities~\ref{Hg2-upper} and~\ref{Hg2-lower} are
inconsistent once $N > \frac{\trdeg_\CC K_0(\widetilde{a}_1,\widetilde{a}_2) - 1}{\dim \bG_2}$.  
Thus, the hypothesis that $(a_1,a_2)$ is Hodge generic
in $S_1 \times S_2$ must be wrong and we find a proper
weakly special variety $T \subseteq S_1 \times S_2$
with $(a_1,a_2) \in T(L)$. Since each of $a_i$ is 
individually Hodge generic in $S_i$ 
(for $i = 1$ or $2$), $T$ is strongly special.  
That the projections 
$T \to S_i$ are finite follows from 
irreducibility of $S_1$ and $S_2$.
\end{proof}

We derive several consequences from 
Lemma~\ref{dependence-lem}.

\begin{cor}
Let $S$ be a connected.   Let $K \subseteq L$ be an
extension of differential fields each with  
with field of constants $\mathbb{C}$ and $a \in S(L)$ an $L$-valued 
point of $S$ which is Hodge generic 
and not algebraic over
$K$.  Then forking defines a trivial 
pregeometry on $\tp(a/K)$. 	
\end{cor}
\begin{proof}
We will show by induction on $N$ that if 
there is a dependence over $K$ on a sequence 
$a_1, \ldots, a_N$ of realizations of 
$\tp(a/K)$, then there is a dependence between 
$a_i$ and $a_j$ for some $i < j \leq N$. 
For $N \leq 2$, this is trivial.  Consider
the inductive case of $N+1$.    
If $\{a_1, \ldots, a_N\}$ or $\{a_1, \ldots,
a_{N-1},a_{N+1} \}$ are dependent, then by 
induction we already find a pairwise dependence. 
If both of these sequences are independent, 
then let $M := K \langle a_1, 
\ldots, a_{N-1} \rangle$ be the differential 
field generated by $a_1, \ldots, a_{N-1}$ over 
$K$.  In this case, each of $\tp(a_N/M)$ 
and $\tp(a_{N+1}/M)$ is the nonforking 
extension of $\tp(a/K)$ to $M$ and $a_N$ and 
$a_{N+1}$ are dependent over $M$.  By 
Lemma~\ref{dependence-lem}, $(a_N,a_{N+1})$ 
lies on a proper special subvariety of $S \times S$, 
so that this pair is dependent over $K$. 
\end{proof}

As a more direct consequence of 
Lemma~\ref{dependence-lem}, we see that nonorthogonality
comes only from Hecke correspondences.

\begin{cor}
Let $S_1$ and $S_2$ be connected, 
irreducible, pure Shimura varieties.  
Let $K \subseteq L$ be an
extension of differential fields each with  
with field of constants $\mathbb{C}$ and 
$a_i \in S_i(L)$ for $i = 1$ or $2$ be Hodge 
generic points each of which is not algebraic 
over $K$. Then $\tp(a_1/K) \not \perp 
\tp(a_2/K)$ if and only if there is a Shimura
variety $T$, finite maps of Shimura varieties
$\nu_i:T \to S_i$, and a point $b \in T(L^\text{alg})$
for which $\tp(\nu_i(b)/K) = \tp(a_i/K)$ for
$i = 1$ and $2$.	
\end{cor}
\begin{proof}
If $\tp(a_1/K) \not \perp \tp(a_2/K)$, then we
can find some extension $M$ of $K$ and points 
$c_1$ and $c_2$ with $\tp(c_i/M)$ being the 
nonforking extension of $\tp(a_i/K)$ to 
$M$ (for $i = 1$ and $2$) and $c_1$ and $c_2$
are dependent over $M$.  By Lemma~\ref{dependence-lem},
there is some proper special 
$T \subseteq S_1 \times S_2$ with $(c_1,c_2) \in T$
and $T \to S_i$ finite for $i = 1$ and $2$.  
Set $b := (c_1,c_2)$. 	
\end{proof}

\bibliographystyle{siam}
\bibliography{ZP}

 \end{document}